\documentclass[12pt]{article}
\usepackage{amssymb}
\usepackage{amsthm}
\usepackage{amsmath}
\usepackage{mathrsfs}
\usepackage{verbatim}

\begin{document}
\newcommand{\hess}[1]{\ensuremath{Hess_{x_{#1}}}}

\title{Uniqueness and Monge solutions in the multi-marginal optimal transportation problem\footnote{The author was supported in part by an NSERC postgraduate scholarship.  This work was completed in partial fulfillment of the requirements for a doctoral degree in mathematics at the University of Toronto.}} \author{BRENDAN PASS \footnote{Department of Mathematics, University of Toronto, Toronto, Ontario, Canada, M5S 2E3 bpass@math.utoronto.ca.}}
\maketitle

\begin{abstract}We study a multi-marginal optimal transportation problem.  Under certain conditions on the cost function and the first marginal, we prove that the solution to the relaxed, Kantorovich version of the problem induces a solution to the Monge problem and that the solutions to both problems are unique.
\end{abstract}
\section{Introduction}

Given two Borel probability measures $\mu_1$ and $\mu_2$ on spaces $M_1$ and $M_2$ respectively and a cost function $c:M_1 \times M_2 \rightarrow \mathbb{R}$, Monge's optimal transportation problem asks how to most efficiently map 
$\mu_1$ onto $\mu_2$, where efficiency is measured relative to $c$.  We say that a function $G: M_1 \rightarrow M_2$ pushes $\mu_1$ forward to $\mu_2$ and write $G_{\#}\mu_1=\mu_2$ if $\mu_1(G^{-1}(B))=\mu_2(B)$ for all measurable $B \subseteq M_2$.  Monge's problem is then to minimize $\int_{M_1} c(x_1,G(x_1))d\mu_1$ among all $G$ that push $\mu_1$ forward to $\mu_2$.  Due to both its deep connections to other areas of mathematics and its applicability in other fields, optimal transportation has grown into a thriving field of research over the past 20 years; we refer the interested reader to the books by Villani for references \cite{V}\cite{V2}.

In this paper, we are interested in a multi-marginal version of the preceding problem.  Given Borel probability measures $\mu_i$ on $n$-dimensional, smooth manifolds $M_i$, for $i=1,2,...,m$, and a \emph{cost} function $c:M_1 \times M_2\times.... \times M_m\rightarrow \mathbb{R}$, the multi-marginal version of Monge's optimal transportation problem is to minimize: 
\begin{equation*}
C(G_2,G_3,...,G_m):=\int_{M_1}c(x_1,G_2(x_1),G_3(x_1),...,G_m(x_1))d\mu_1 \tag{\textbf{M}}
\end{equation*}
among all $(m-1)$-tuples of measurable maps $(G_2,G_3,...,G_m)$, where $G_i: M_1 \rightarrow M_i$ pushes $\mu_1$ forward to $\mu_i$ for all $i=2,3,...,m$.
The Kantorovich formulation of the problem is to minimize:
\begin{equation*}
C(\mu):=\int_{M_1 \times M_2\times ... \times M_m}c(x_1,x_2,x_3,...,x_m)d\mu \tag{\textbf{K}}
\end{equation*}
among all positive Borel measures $\mu$ on $M_1 \times M_2 \times ...\times M_m$ such that the canonical projection
\begin{equation*}
\pi_i:M_1 \times M_2 \times...\times M_m \rightarrow M_i
\end{equation*}
pushes $\mu$ forward to $\mu_i$ for all $i$.  For any $(m-1)$-tuple $(G_2,G_3,...,G_m)$ such that $G_{i\#}\mu_1=\mu_i$ for all $i=2,3,...,m$, we can define the measure $\mu=(Id,G_2,G_3,..G_M)_{\#}\mu_1$ on $M_1 \times M_2 \times...\times M_m$, where $Id:M_1 \rightarrow M_1$ is the identity map.  Then $\mu$ projects to $\mu_i$ for all $i$ and $C(G_2,G_3,...,G_m)=C(\mu)$; therefore, \textbf{K} can be interpreted as a relaxed version of \textbf{M}.  Roughly speaking, the difference between the two formulations is that in \textbf{M} almost every point $x_1 \in M_1$ is coupled with exactly one point $x_i \in M_i$ for each $i=2,3,...,m$, whereas in \textbf{K} an element of mass at $x_i$ is allowed to be \emph{split} between two or more target points in $M_i$ for $i=2,3,...,m$.

Assuming that $c$ is continuous, it is not hard to show that a solution to \textbf{K} exists.  When $m=2$, under a regularity condition on $\mu_1$ and a twist condition on $c$, which we will define in the next section, one can show that this solution is concentrated on the graph of a function over $x_1$ \cite{lev}\cite{g}\cite{bren}\cite{gm}\cite{Caf}.   It is then straightforward to show that this function solves \textbf{M} and to establish uniqueness results for both \textbf{M} and \textbf{K}.  When $m \geq 3$, however, existence and uniqueness in \textbf{M} as well as uniqueness in \textbf{K} are still largely open. In their seminal paper, Gangbo and \'Swi\c ech \cite{GS} used a duality theorem of Kellerer \cite{K} to resolve these questions for the quadratic cost function, $c(x_1,x_2,x_3,...,x_m) = -|\sum_{i=1}^m x_i|^2$ on $\mathbb{R}^n \times \mathbb{R}^n \times....\times\mathbb{R}^n$, extending partial results for the same cost by Olkin and Rachev \cite{OR}, Knott and Smith \cite{KS} and R\"uschendorf and Uckelmann \cite{RU}.  Gangbo and \'Swi\c ech's theorem was then reproved using a different argument by R\"uschendorf and Uckelmann\cite{RU2} and generalized by Heinich \cite{H} to cost functions of the form $c(x_1,x_2,x_3,...,x_m) = h(\sum_{i=1}^m x_i)$ where $h: \mathbb{R}^n \rightarrow \mathbb{R}$ is strictly concave.  Carlier extended these results to a still wider class of cost functions, but only in the case when all the domains $M_i$ are one dimensional \cite{C}.  A result with a somewhat different flavour was obtained by Carlier and Nazaret \cite{CN}.  Assuming $n=m$ and $M_i=\mathbb{R}^n$ for all $i$, their objective was to maximize a convex function of the determinant of the matrix whose columns are the vectors $x_i$; for this cost function, they demonstrate that the maximizer may not be concentrated on the graph of a function over $x_1$ and may not be unique.  The aim of the present article is to identify general conditions on $c$ under which both \textbf{K} and \textbf{M} admit unique solutions.  

With one exception, the conditions we impose will look similar to standard conditions which arise when studying the two marginal problem.  Our lone novel hypothesis is that a certain covariant $2$-tensor on the product space $M_2 \times M_2\times... \times M_{m-1}$ should be negative definite.  Tensors have recently become very relevant to optimal transportation, owing to their essential role in the study of the regularity of optimal maps.  Ma, Trudinger and Wang \cite{MTW} showed that a certain tensorial condition dictates the regularity of the solution to \textbf{M} in the two marginal case, and Kim and McCann \cite{KM} reinterpreted this condition, relating it to the sectional curvature of a certain pseudo-metric.  In \cite{P}, the present author showed that the dimension of the support of the optimizers for multi-marginal problems is related to a family of symmetric, bi-linear forms.  One surprising consequence of that work is a class of counterexamples demonstrating that the obvious generalization of the twist condition to the multi-marginal setting is sufficient neither to guarantee uniqueness of minimizers for \textbf{K} nor to ensure that the solution to \textbf{K} induces a solution to \textbf{M}.  In these examples, the solutions were concentrated on submanifolds of dimension greater than $n$; motivated in part by this observation, \cite{P} identified local conditions on $c$ under which the support of the optimal measure is at most $n$-dimensional.  Our condition here is a little different.  Whereas the question about the dimension of the support of a solution $\mu$ to \textbf{K} is purely local, showing that $\mu$ gives rise to a solution to \textbf{M} is a global issue: for almost all $x_1 \in M_1$ we must show that there is exactly one $(x_2,x_3,...,x_m) \in M_2 \times M_3 \times,...,M_m$ which get coupled to $x_1$ by $\mu$.  Our tensor here is designed to capture this global nature of the problem.

In the next section we recall relevant concepts from the theory of optimal transportation and formulate the conditions we will need.  In section 3 we state and prove our main result and in the fourth section we exhibit several examples of cost functions which satisfy the criteria of our main theorem.

\textbf{Acknowledgment:}  It is a pleasure to thank my advisor, Robert McCann, for his support and for many useful discussions throughout the duration of this work.

\section{Preliminaries and definitions}

We will assume that each $M_i$ can be smoothly embedded in some larger manifold in which its closure $\overline{M_i}$ is compact and that the cost $c \in C^2(\overline{M_1}\times\overline{M_2}\times...\times\overline{M_m})$.  In addition, we will assume that $M_i$ is a Riemannian manifold for $i=2,3,..m-1$ and that any two points can be joined by a smooth, length minimizing geodesic\footnote{Note that we do \emph{not} assume $M_i$ is complete, however, as we do not wish to exclude, for example, bounded, convex domains in $\mathbb{R}^n$.}, although no such assumptions will be needed on $M_1$ or $M_m$.  The requirement of a Riemannian structure is related to the global nature of $\textbf{M}$ that we alluded to in the introduction; a Riemannian metric gives us a natural way to connect any pair of points, namely geodesics. 

We will denote by $D_{x_i}c(x_1,x_2,...,x_m)$ the differential of $c$ with respect to $x_i$. For $i \neq j$, the bi-linear form $D^2_{x_ix_j}c(x_1,x_2,...,x_m)$ on $T_{x_i}M_i \times T_{x_j}M_j$ was introduced in $\cite{KM}$; in local coordinates, it is defined by 
\begin{equation*}
D^2_{x_ix_j}c\langle\frac{\partial}{\partial x_{i}^{\alpha_{i}}}, \frac{\partial}{\partial x_{j}^{\alpha_{j}}}\rangle =\frac{\partial ^{2} c}{\partial x_{i}^{\alpha_{i}}\partial x_{j}^{\alpha_{j}}}.
\end{equation*}
As $M_i$ is Riemannian for $i=2,...,m-1$, Hessians or unmixed, second order partial derivatives with respect to these coordinates make sense and we will denote them by $\hess{i}c(x_1,x_2,...,x_m)$; note, however, that no Riemannian structure is necessary to ensure the tensoriality of the mixed second order partials $D^2_{x_ix_j}c(x_1,x_2,...,x_m)$, as was observed in \cite{KM}.

Given a Borel probability measure $\mu$ on a topological space $M$, the support of $\mu$, which we will denote by $spt(\mu)$, is defined to be the smallest closed subset of $M$ such that $\mu(spt(\mu))=1$.

The dual problem to \textbf{K} is to maximize 
\begin{equation*}
\sum^m_{i=1}\int_{M_i} u_i(x_i)d\mu_i \tag{\textbf{D}}
\end{equation*}
among all $m$-tuples $(u_1,u_2,...,u_m)$ of functions $u_i \in L^1(\mu_i)$ for which $\sum_{i=1}^m u_i(x_i ) \leq c(x_1,...,x_m)$ for all $(x_1,...,x_m) \in M_1 \times M_2 \times ...\times M_m$.

There is a special class of functions satisfying the constraint in \textbf{D} that will be of particular interest to us:
\newtheorem{ccon}{Definition}[section]
\begin{ccon}
We say that an $m$-tuple of functions $(u_1,u_2,..u_m)$ is $c$-conjugate if for all $i$
\begin{equation*}
u_i(x_i)=\inf_{\substack{x_j \in M_j\\ j\neq i}}\Big( c(x_1,x_2,...,x_m)-\sum_{j \neq i}u_j(x_j)\Big)
\end{equation*}
\end{ccon}

Whenever $(u_1,u_2,..u_m)$ is $c$-conjugate, the $u_i$ are semi-concave and hence have super differentials  $\overline{\partial} u_i(x_i)$ at each point $x_i \in M_i$.  By compactness, for each $x_i \in M_i$ we can find $x_j \in \overline{M_j}$ for all $j \neq i$ such that $u(x_i)=c(x_1,x_2,...,x_m)-\sum_{j \neq i}u_j(x_j)$; furthermore, as long as $|u_i(x_i)|< \infty$ for at least one $x_i$, $u_i$ is locally Lipschitz \cite{M}. 
 
The following theorem makes explicit the link between the Kantorovich problem and its dual.

\newtheorem{duality}{Theorem}[section]

\begin{duality}\label{du}
There exists a solution $\mu$ to the Kantorovich problem and a $c$-conjugate solution $(u_1,u_2,...,u_m)$ to its dual.  Furthermore, the maximum value in \textbf{D} coincides with the minimum value in \textbf{K}.  Finally, for any solution $\mu$ to \textbf{K}, any $c$-conjugate solution $(u_1,u_2,...,u_m)$ to \textbf{D} and any $(x_1,...,x_m) \in spt(\mu)$ we have $\sum_{i=1}^mu_i(x_i) =c(x_1,...,x_m)$.
\end{duality}
This result is well known in the two marginal case; for $m \geq 3$, the existence of solutions to \textbf{K} and \textbf{D} as well as the equality of their extremal values was proved in \cite{K}.  The remaining conclusions were proved for a special cost by Gangbo and \'Swi\c ech \cite{GS} and for a general, continuous cost when each $M_i=\mathbb{R}^n$ by Carlier and Nazaret \cite{CN}.  The same proof applies for more general spaces $M_i$; we reproduce it below in the interest of completeness.

\begin{proof}
As mentioned above, a proof of the existence of solutions $\mu$ to \textbf{K} and $(v_1,v_2,...,v_m)$ to \textbf{D} as well as the equality:
\begin{equation} \label{sol}
 \sum^m_{i=1}\int_{M_i} v_i(x_i)d\mu_i=\int_{M_1 \times M_2 ....\times M_m}c(x_1,x_2,x_3,...,x_m)d\mu 
\end{equation}
can be found in \cite{K}.  We use a convexification trick, also found in \cite{GS} and \cite{CN}, to build a $c$-conjugate solution to \textbf{D}.

Define 
\begin{equation*}
 u_1(x_1)=\inf_{\substack{x_j \in M_j \\ j \geq 2}} \Big(c(x_1,x_2,...,x_m) -\sum_{j=2}^{m} v_j(x_j)\Big)
\end{equation*}
and $u_i$ inductively by 
\begin{equation*}u_i(x_i)=\inf_{\substack{x_j \in M_j \\ j\neq i}} \Big(c(x_1,x_2,...,x_m) -\sum_{j=1}^{i-1} u_j(x_j) - \sum_{j =i+1}^m v_j(x_j)\Big)
 \end{equation*}
As 
\begin{equation*}
 u_m(x_m) = \inf_{\substack{x_j \in M_j \\ j\neq i}} \Big(c(x_1,x_2,...,x_m) -\sum_{j=1}^{m-1} u_j(x_j)\Big),
\end{equation*}
 we immediately obtain 
\begin{equation}\label{in1}
 u_i(x_i) \leq \inf_{\substack{x_j \in M_j \\ j\neq i}}\Big( c(x_1,x_2,...,x_m)-\sum_{j \neq i}u_j(x_j)\Big).
\end{equation}
The definition of $u_{i-1}$ implies that for all $(x_1,x_2,...x_m)$
\begin{equation*}
 v_i(x_i) \leq c(x_1,x_2,...,x_m) -\sum_{j=1}^{i-1} u_j(x_j) - \sum_{j =i+1}^m v_j(x_j)
\end{equation*}
Therefore, $v_i(x_i) \leq u_i(x_i)$.   It then follows that 

\begin{eqnarray*}
u_i(x_i) &=& \inf_{\substack{x_j \in M_j \\ j\neq i}} \Big(c(x_1,x_2,...,x_m) -\sum_{j=1}^{i-1} u_j(x_j) - \sum_{j =i+1}^m v_j(x_j) \Big) \\
& \geq & \inf_{\substack{x_j \in M_j \\ j\neq i}} \Big(c(x_1,x_2,...,x_m) -\sum_{j \neq i} u_j(x_j)\Big),
\end{eqnarray*}
which, together with (\ref{in1}), implies that $(u_1,u_2,...,u_m)$ is $c$-conjugate.  Now, we have  
\begin{eqnarray*}
 \sum^m_{i=1}\int_{M_i} v_i(x_i)d\mu_i &\leq& \sum^m_{i=1}\int_{M_i} u_i(x_i)d\mu_i\\
&=& \sum^m_{i=1}\int_{M_1 \times M_2 ....\times M_m} u_i(x_i)d\mu \\
& \leq & \int_{M_1 \times M_2 ....\times M_m}c(x_1,x_2,x_3,...,x_m)d\mu
\end{eqnarray*}
and so by (\ref{sol}) we must have

\begin{equation*} \label{c-sol}
 \sum^m_{i=1}\int_{M_i} u_i(x_i)d\mu_i=\sum^m_{i=1}\int_{M_1 \times M_2 ....\times M_m} u_i(x_i)d\mu=\int_{M_1 \times M_2 ....\times M_m}c(x_1,x_2,x_3,...,x_m)d\mu 
\end{equation*}

But because $\sum^m_{i=1} u_i(x_i) \leq c(x_1,x_2,x_3,...,x_m)$, we must have equality $\mu$ almost everywhere.  Continuity then implies equality holds on $spt(\mu)$.
\end{proof}
As a corollary to the duality theorem, we now prove a uniqueness result for the solution to \textbf{D}.  When $m=2$, this result, under the weak conditions on $c$ stated below, is due to Chiappori, McCann and Nesheim \cite{CMN}; for certain special, multi-marginal costs, it was proven by Gangbo and \'Swi\c ech \cite{GS} and Carlier and Nazaret \cite{CN}.  Although this result is tangential to the main goals of this article, we prove it here to emphasize that, whereas  uniqueness in \textbf{K} requires certain structure conditions on the cost, uniqueness in \textbf{D} depends only on the differentiability of $c$.

\newtheorem{unD}{Corollary}[section]
\begin{unD}
 Suppose the domains $M_i$ are all connected, that $c$ is continuously differentiable and that each $\mu_i$ is absolutely continuous with respect to local coordinates with a strictly positive density.  If $(v_1,v_2,...,v_m)$ and $(\overline{v_1},\overline{v_2},...,\overline{v_m})$ solve \textbf{D}, then there exist constants $t_i$ for $i=1,2.,,,m$ such that $\sum_{i=1}^mt_i=0$ and $v_i=\overline{v_i}+t_i$, $\mu_i$ almost everywhere, for all $i$.
\end{unD}
\begin{proof}
 Using the convexification trick in the proof of Theorem \ref{du}, we can find $c$-conjugate solutions $(u_1,u_2,...,u_m)$ and $(\overline{u_1},\overline{u_2},...,\overline{u_m})$ to \textbf{D} such that $v_i(x_i) \leq u_i(x_i)$ and $\overline{v_i}(x_i) \leq \overline{u_i}(x_i)$ for all $x_i \in M_i$.  Now, as 
\begin{equation*}
 \sum^m_{i=1}\int_{M_i} v_i(x_i)d\mu_i = \sum^m_{i=1}\int_{M_i} u_i(x_i)d\mu_i
\end{equation*}
we must have $v_i=u_i$, $\mu_i$ almost everywhere.  Similarly, $\overline{v_i}=\overline{u_i}$, $\mu_i$ almost everywhere.  Now, choose $x_i \in M_i$ where $u_i$ and $\overline{u_i}$ are differentiable.  Then there exists $x_j$ for all $j \neq i$ such that 
\begin{equation*}
 (x_1,x_2,...,x_{i-1},x_i,x_{i+1}...,x_m) \in spt(\mu);
\end{equation*}
Theorem \ref{du} then yields 
\begin{equation*}
 u_i(x_i) -c(x_1,x_2,...,x_{i-1},x_i,x_{i+1}...,x_m)= -\sum_{j \neq i}u_j(x_j).
\end{equation*}
  Because 
\begin{equation*}
 u_i(z_i) - c(x_1,x_2,...,x_{i-1},z_i,x_{i+1}...,x_m) \leq -\sum_{j \neq i}u_j(x_j)
\end{equation*}
 for all other $z_i \in M_i$ we must have 
\begin{equation*}
 Du_i(x_i) =D_{x_i}c(x_1,x_2,...,x_{i-1},x_i,x_{i+1}...,x_m). 
\end{equation*}
 Similarly, 
\begin{equation*}
 D\overline{u_i}(x_i) =D_{x_i}c(x_1,x_2,...,x_{i-1},x_i,x_{i+1}...,x_m),
\end{equation*}
hence $Du_i(x_i)=D\overline{u_i}(x_i)$.  As this equality holds for almost all $x_i$ we conclude $u_i(x_i)=\overline{u_i}(x_i)+t_i$ for some constant $t_i$.  Choosing any $(x_1,x_2,...,x_m) \in \text{spt}(\mu)$ and noting that  
\begin{equation*}
 \sum_{i=1}^mu_i(x_i) =c(x_1,x_2,...,x_m)=\sum_{i=1}^m\overline{u_i}(x_i),
\end{equation*}
we obtain $\sum_{i=1}^mt_i=0$.
\end{proof}
 
The next two definitions are straightforward generalizations of concepts borrowed from the two marginal setting.

\newtheorem{twist}[duality]{Definition}
\begin{twist}
For $i \neq j$, we say that $c$ is $(i,j)$-twisted if  the map $x_j \in M_j \mapsto D_{x_i}c(x_1,x_2,...,x_m) \in T_{x_i}^{*}M_i$ is injective, for all fixed $x_k$, $k \neq j$.
\end{twist}

\newtheorem{non-degeneracy}[duality]{Definition}
\begin{non-degeneracy}
We say that $c$ is $(i,j)$-non-degenerate if $D^2_{x_ix_j}c(x_1,x_2,...,x_m)$, considered as a map from $T_{x_j}M_j$ to $T^*_{x_i}M_i$, is injective for all $(x_1,x_2,...,x_m)$.  
\end{non-degeneracy}

In local coordinates, non-degeneracy simply means that the corresponding matrix of mixed, second order partial derivatives has a non-zero determinant.  When this condition holds, the inverse map $T^*_{x_i}M_i \rightarrow T_{x_j}M_j$ will be denoted by $(D^2_{x_ix_j}c)^{-1}(x_1,x_2,...,x_m)$.

When $m=2$, the non-degeneracy condition is not needed to ensure the existence of an optimal map (although it plays an important role in studying the regularity of that map).  On the other hand, the twist condition plays an essential role in showing that Monge's problem has a solution; it ensures that a first order, differential condition arising from the duality theorem can be solved uniquely for one variable as a function of the other \cite{lev} (see also \cite{bren}, \cite{gm} and \cite{Caf}).  In light of this, one might expect that, for $m \geq 3$, if $c$ is $(i,j)$-twisted for all $i \neq j$, then the Kantorovich solution $\mu$ induces a Monge solution.  This is not true, as our example in \cite{P} demonstrates.  In the multi-marginal problem, duality yields $m$ first order conditions; our strategy in this paper is to show that if we fix the first variable, these equations can be uniquely solved for the other $m-1$ variables.  In the problems considered by Gangbo and \'Swi\c ech \cite{GS} and Heinich \cite{H}, these equations turn out to have a particularly simple form and can be solved explicitly.  For more general cost functions, this becomes a much more subtle issue.  Our proof will combine a second order, differential condition with tools from convex analysis and will require that the tensor $T$, defined below, is negative definite.

\newtheorem{tensor}[duality]{Definition}

\begin{tensor}
Suppose $c$ is $(1,m)$-non-degenerate.  Let $\vec{y}=(y_1,y_2,...,y_m) \in M_1 \times M_2 \times...\times M_m$.  For each $i:=2,3,...,m-1$ choose a point $\vec{y}(i)=(y_1(i),y_2(i),...,y_m(i)) \in \overline{M_1} \times \overline{M_2} \times...\times \overline{M_m}$ such that $y_i(i)=y(i)$.  Define the following bi-linear maps on $T_{y_2}M_2 \times T_{y_3}M_3 \times ...\times T_{y_{m-1}}M_{m-1}$:

\begin{equation*}
 S_{\vec{y}}=-\sum_{j=2}^{m-1} \sum_{\substack {i=2 \\ i \neq j}}^{m-1}D^2_{x_ix_j}c(\vec{y}) +\sum_{i,j=2}^{m-1}(D^2_{x_ix_m}c(D^{2}_{x_1x_m}c)^{-1}D^2_{x_1x_j}c)(\vec{y})
\end{equation*}

\begin{equation*}
 H_{\vec{y},\vec{y}(2),\vec{y}(3),...,\vec{y}(m-1)}=\sum_{i=2}^{m-1}(\hess{i}c(\vec{y}(i))-\hess{i}c(\vec{y}))
\end{equation*}

\begin{eqnarray*}
T_{\vec{y},\vec{y}(2),\vec{y}(3),...,\vec{y}(m-1)}=S_{\vec{y}}+H_{\vec{y},\vec{y}(2),\vec{y}(3),...,\vec{y}(m-1)}
\end{eqnarray*}
\end{tensor}
Note that $D^2_{x_ix_j}c(x_1,x_2,...,x_m), \hess{i}c(x_1,x_2,...,x_m)$ and the composition $\big(D_{x_ix_m}c(D^{2}_{x_1x_m}c)^{-1}D^2_{x_1x_j}c\big)(x_1,x_2,...,x_m)$ are actually bi-linear maps on the spaces $T_{x_i}M_i \times T_{x_j}M_j$, $T_{x_i}M_i \times T_{x_i}M_i$ and $T_{x_i}M_i \times T_{x_j}M_j$, respectively, but we can extend them to maps on the product space $(T_{x_2}M_2 \times T_{x_3}M_3 \times ...\times T_{x_{m-1}}M_{m-1})^2$ by considering only the appropriate components of the tangent vectors.  

Though $T$ looks complicated, it appears naturally in our argument.  The condition $T<0$ is in one sense analogous to the twist and non-degeneracy conditions that are so important in the two marginal problem.  Like the non-degeneracy condition, negativity of $S$ is an inherently local property on $M_1 \times M_2 \times ...\times M_m$; under this condition, one can show that our system of equations is locally uniquely solvable.  To show that the solution is actually globally unique requires something more; in the two marginal case, this is the twist condition, which can be seen as a global extension of non-degeneracy.  In our setting, requiring that the sum $T=S+H<0$ turns out to be enough to ensure that the locally unique solution is in fact globally unique.

\section{Monge solutions}

\newtheorem{main}{Theorem}[section]
We are now in a position to precisely state our main theorem:
\begin{main}\label{mon}
Suppose that:
\begin{enumerate}
\item $c$ is $(1,m)$-non-degenerate. 
\item $c$ is $(1,m)$-twisted.
\item For all choices of $\vec{y}=(y_1,y_2,...,y_m) \in M_1 \times M_2 \times...\times M_m$ and of $\vec{y}(i)=(y_1(i),y_2(i),...,y_m(i)) \in \overline{M_1}\times \overline{M_2} \times...\times \overline{M_m}$ such that $y_i(i)=y_i$ for $i=2,...,m-1$, we have 

\begin{equation}
T_{\vec{y},\vec{y}(2),\vec{y}(3),...,\vec{y}(m-1)} < 0. 
\end{equation}
\item The first marginal $\mu_1$ does not charge sets of Hausdorff dimension less than or equal to $n-1$.
\end{enumerate}
Then any solution $\mu$ to the Kantorovich problem is concentrated on the graph of a function; that is, there exist functions $G_i:M_1 \rightarrow M_i$ such that $\text{graph}(\vec{G})=\{(x_1,G_2(x_1),G_3(x_1),...,G_m(x_1))\}$ satisfies $\mu(\text{graph}(\vec{G}))=1$ 
\end{main}

\begin{proof}

Let $u_i$ be a $c$-conjugate solution to the dual problem.  Now, $u_1$ is semi-concave and hence differentiable off a set of Hausdorff dimension $n-1$; as $\mu_1$ vanishes on every set of Hausdorff dimension less than or equal to $n-1$, by Theorem \ref{du} it suffices to show that for every $x_1 \in M_1$ where $u_1$ is differentiable, there is at most one $(x_2,x_3,...,x_m) \in M_2 \times M_3 \times ...\times M_m$ such that $\sum_{i=1}^m u_i(x_i)=c(x_1,x_2,x_3,...,x_m)$. Note that this equality implies that $D_{x_i}c(x_1,x_2,...,x_m) \in \overline{\partial} u_i(x_i)$ for all $i=1,2...,m$; in particular, as $u_1$ is differentiable at $x_1$, $Du_1(x_1)=D_{x_1}c(x_1,x_2,...,x_m)$.  Our strategy will be to show that these inclusions can hold for at most one $(x_2,x_3,...,x_m)$.

Fix a point $x_1$ where $u_1$ is differentiable.  Twistedness implies that the equation $Du_1(x_1)=D_{x_1}c(x_1,x_2,...,x_m)$ defines $x_m$ as a function $x_m=F_{x_1}(x_2,...,x_{m-1})$ of the variables $x_2,x_3,...,x_{m-1}$; non-degeneracy and the implicit function theorem then imply that $F_{x_1}$ is continuously differentiable with respect to $x_2,x_3,...,x_{m-1}$ and
\begin{equation*} 
D_{x_i}F_{x_1}(x_2,...,x_{m-1})=-\big((D^{2}_{x_1x_m}c)^{-1}D^2_{x_1x_i}c)(x_1,x_2,...,F_{x_1}(x_2,...,x_{m-1})\big) 
\end{equation*}
for $i=2,...,m-1$.  We will show that there exists at most one point $(x_2,x_3,...,x_{m-1}) \in M_2 \times M_3 \times...\times M_{m-1}$ such that 
\begin{equation*}
D_{x_i}c(x_1,x_2,...,F_{x_1}(x_2,..x_{m-1})) \in \overline{\partial} u_i(x_i) 
\end{equation*}
for all $i=2,...,m-1$.

The proof is by contradiction; suppose there are two such points, $(x_2,x_3,...,x_{m-1})$ and $(\overline{x_2},\overline{x_3},...,\overline{x_{m-1}})$.  For $i=2,...,m-1$, we can choose Riemannian geodesics $\gamma_i(t)$ in $M_i$ such that $\gamma_i(0)=x_i$ and $\gamma_i(1)=\overline{x_i}$.  Take a measurable selection of covectors $V_i(t) \in \partial u_i(\gamma_i(t))$.  We will show that $f(1) <f(0)$, where
\begin{equation*}
f(t):=\sum_{i=2}^{m-1}[V_i(t)-D_{x_i}c(x_1,\vec{\gamma}(t)]\langle\frac{d\gamma_i}{dt}\rangle
\end{equation*}
and we have used $(x_1, \vec{\gamma}(t))$ as a shorthand for $(x_1,\gamma_2(t) ,...,\gamma _{m-1}(t),F_{x_1}(\gamma_2(t),...,\gamma_{m-1}(t)))$ and $a<b>$ to denote denote the duality pairing between a $1$-form $a$ and a vector $b$.  This will clearly imply the desired result.

For each $t$ and each $i=2,...,m-1$, by $c$-conjugacy of $u_i$ and the compactness of $\overline{M_j}$ for $j \neq i$, we can choose $\vec{y}(i;t)=(y_1(i;t),y_2(i;t),...,y_m(i:t)) \in \overline{M_1} \times \overline{M_2}\times...\times \overline{M_m}$ so that $y_i(i;t)=\gamma_i(t)$ and 
\begin{equation*}
\sum_{j=1}^{m} u_j(y_j(i;t))=c(y_1(i;t),y_2(i;t),...,y_m(i;t))  
\end{equation*}
Note that $V_i(t)\langle\frac{d\gamma_i}{dt}\rangle$ supports the semi-concave function $T \in [0,1] \mapsto u_i(\gamma_i(t))$.  But $u_i(\gamma_i(t))$ is twice differentiable almost everywhere and hence we have $V_i(t)\langle\frac{d\gamma_i}{dt}\rangle= \frac{d(u_i(\gamma_i(t)))}{dt}$ for almost all $t$ and, by semi-concavity, $V_i(1)\langle\frac{d\gamma_i}{dt}\rangle-V_i(0)\langle\frac{d\gamma_i}{dt}\rangle \leq \int_{0}^{1} \frac{d^2(u_i(\gamma_i(t)))}{dt^2}dt$.  Now, for any $t,s\in[0,1]$
\begin{equation*}
u_i(\gamma_i(t)) \leq c(y_1(i;s),y_2(i;s),...,y_{i-1}(i;s),\gamma_i(t),y_{i+1}(i;s)...,y_m(i;s)) -\sum_{j \neq i} u_j(y_j(i;s))
\end{equation*} 
and we have equality when $t=s$, as $\gamma_i(s)=y_i(i;s)$.  Hence, whenever $\frac{d^2(u_i(\gamma_i(t)))}{dt^2}$ exists, we have 
\begin{eqnarray*}
\frac{d^2(u_i(\gamma_i(t)))}{dt^2}\Bigg|_{t=s} &\leq& \frac{d^2(c(y_1(i;s),y_2(i;s),...,y_{i-1}(i;s),\gamma_i(t),y_{i+1}(i;s)...,y_m(i;s)) )}{dt^2}\Bigg|_{t=s}\\
&=&\hess{i}c(y_1(i;s),y_2(i;s),...,y_m(i;s))\langle\frac{d\gamma_i}{ds},\frac{d\gamma_i}{ds}\rangle
\end{eqnarray*}
We conclude that 
\begin{equation}\label{subdif}
V_i(1)\langle\frac{d\gamma_i}{dt}\rangle-V_i(0)\langle\frac{d\gamma_i}{dt}\rangle \leq \int_0^1\hess{i}c(y_1(i;t),y_2(i;t),...,y_m(i;t))\langle\frac{d\gamma_i}{dt},\frac{d\gamma_i}{dt}\rangle dt
\end{equation}

Turning now to the other term in $f(1)-f(0)$, we have 
\begin{eqnarray}
&&D_{x_i}c(x_1,\vec{\gamma(1)})\langle\frac{d\gamma_i}{dt}\rangle-D_{x_i}c(x_1,\vec{\gamma(0)})\langle\frac{d\gamma_i}{dt}\rangle \nonumber \\
&=&\int_0^1 \frac{d}{dt}\big( D_{x_i}c(x_1,\vec{\gamma}(t))\langle\frac{d\gamma_i}{dt}\rangle \big) dt \nonumber \\
&=& \int_0^1 \Bigg( \sum_{\substack {j=2 \\ j\neq i}}^{m-1}\Big( D^2_{x_ix_j}c(x_1,\vec{\gamma}(t))\Big)\langle\frac{d\gamma_i}{dt},\frac{d\gamma_j}{dt}\rangle\ +\hess{i}c(x_1,\vec{\gamma}(t))\langle\frac{d\gamma_i}{dt},\frac{d\gamma_i}{dt}\rangle\nonumber \\
& &+\sum_{\substack {j=2}}^{m-1}\Big(D^2_{x_ix_m}c(x_1,\vec{\gamma}(t))D_{x_j}F_{x_1}(\vec{\gamma}(t))\Big)\langle\frac{d\gamma_i}{dt},\frac{d\gamma_j}{dt}\rangle \Bigg)dt\nonumber\\
&=&\int_0^1 \Bigg ( \sum_{\substack {j=2\\j \neq i}}^{m-1} \Big( D^2_{x_ix_j}c(x_1,\vec{\gamma}(t))\Big)\langle\frac{d\gamma_i}{dt},\frac{d\gamma_j}{dt}\rangle +\hess{i}c(x_1,\vec{\gamma}(t))\langle\frac{d\gamma_i}{dt},\frac{d\gamma_i}{dt}\rangle \nonumber \label{smooth} \\ 
& & -\sum_{\substack {j=2}}^{m-1}\Big((D^2_{x_ix_m}c(D^{2}_{x_1x_m}c)^{-1}D^2_{x_1x_j}c)(x_1,\vec{\gamma}(t))\Big)\langle\frac{d\gamma_i}{dt},\frac{d\gamma_j}{dt}\rangle \Bigg) dt 
\end{eqnarray}

Combining (\ref{subdif}) and (\ref{smooth}) yields
\begin{eqnarray*}
f(1)-f(0) & \leq\int_0^1 T_{\big(x_1,\vec{\gamma}(t)\big),\vec{y}(2;t),\vec{y}(3;t),...,\vec{y}(m-1;t)}\langle \frac{d\vec{\gamma}}{dt},\frac{d\vec{\gamma}}{dt}\rangle dt \\
&<0
\end{eqnarray*}

\end{proof}

\newtheorem{monge}[main]{Corollary}
\begin{monge}
 Under the same conditions as Theorem 1, the Monge problem \textbf{M} admits a unique solution and the solution to the Kantorovich problem \textbf{K} is unique.  
\end{monge}
\begin{proof}
 We first show that the $G_i$ defined in Theorem \ref{mon} push $\mu_1$ to $\mu_i$ for all $i=2,3,..m$.  Pick a Borel set $B \in  M_i$.  We have

\begin{eqnarray*}
 \mu_i(B)&=&\mu(M_1\times M_2 \times...\times M_{i-1} \times B \times M_{i+1} \times ...\times M_m) \\
 &=&\mu\Big((M_1\times M_2 \times...\times M_{i-1} \times B \times M_{i+1} \times ...\times M_m) \cap \text{graph}(\vec{G})\Big)\\
&=& \mu\Big(\{(x_1,G_2(x_1),...,G_m(x_1) | G_i(x_1) \in B\}\Big)\\
&=& \mu\Big((G_i^{-1}(B) \times M_2 \times ...\times M_m) \cap \text{graph}(\vec{G})\Big)\\
&=& \mu (G_i^{-1}(B) \times M_2 \times ...\times M_m )\\
&=& \mu_1(G_i^{-1}(B))
\end{eqnarray*}
  This implies that $(G_2,G_3,...,G_m)$ solves \textbf{M}.  To prove uniqueness of $\mu$, note that any other optimizer $\overline{\mu}$ must also be concentrated on graph$(\vec{G})$, which in turn implies $\overline{\mu}=(\text{Id}, G_2,...,G_m)_{\#}\mu_1=\mu$.  Uniqueness of $(G_2,G_3,..G_m)$ now follows immediately; if $(\overline{G_2},\overline{G_3},...,\overline{G_m})$ is another solution to \textbf{M} then $(\text{Id}, \overline{G_2},\overline{G_3},...,\overline{G_m})_{\#}\mu_1$ is another solution to \textbf{K}, which must then be concentrated on graph$(\vec{G})$.  This means that $G_i=\overline G_i$, $\mu_1$ almost everywhere.
\end{proof}

\section{Examples}

In this section, we discuss several types of cost functions to which Theorem \ref{mon} applies.  In these examples, the complicated tensor $T$ simplifies considerably.

\newtheorem{GanSw}{Example}[section]
\begin{GanSw}(Perturbations of concave functions of the sum)
Gangbo and \'Swi\c ech \cite{GS} and Heinich \cite{H} treated cost functions defined on $(\mathbb{R}^n)^m$ by $c(x_1,x_2,...,x_m)= h(\sum_{k=1}^mx_k)$ where $h: \mathbb{R}^n \rightarrow \mathbb{R}$ is strictly concave.  Here, we make the slightly stronger assumption that $h$ is $C^2$ with $D^{2}h <0$.  Assuming each $\mu_i$ is compactly supported, we can take each $M_i$ to be a bounded, convex domain in $\mathbb{R}^n$.  Now, $D_{x_i}c(x_1,x_2,...,x_m) =Dh(\sum_{k=1}^mx_k)$ and $D^2_{x_ix_j}c(x_1,x_2,...,x_m))=D^2h(\sum_{k=1}^mx_k)$, where we have made the obvious identification between tangent spaces at different points.  $c$ is then clearly $(1,m)$-twisted and $(1,m)$-non-degenerate.  Furthermore, the bi-linear map $S_{\vec{y}}$ on  $(\mathbb{R}^n)^{m-2}$ is block diagonal, and each of its diagonal blocks is 
\begin{equation*}
D^2h(\sum_{k=1}^my_k).
\end{equation*}  
Similarly, as $\hess{i}c(\vec{y}(i))=D^2h(\sum_{k=1}^my_k(i))$ and $\hess{i}c(\vec{y})=D^2h(\sum_{k=1}^my_k)$, $H_{\vec{y},\vec{y}(2),\vec{y}(3),...,\vec{y}(m-1)}$ is block diagonal and its $i$th diagonal block is
\begin{equation*} 
D^2h(\sum_{k=1}^my_k(i))-D^2h(\sum_{k=1}^my_k).
\end{equation*}
Therefore, $T_{\vec{y},\vec{y}(2),\vec{y}(3),...,\vec{y}(m-1)}$ is block diagonal and its $ith$ diagonal block is 
\begin{equation*}
D^2h(\sum_{k=1}^my_k(i)).
\end{equation*}
This is clearly negative definite. Furthermore, $C^2$ perturbations of this cost function will also satisfy $T_{\vec{y},\vec{y}(2),\vec{y}(3),...,\vec{y}(m-1)} < 0$; this shows that the results of Gangbo and \'Swi\c ech and Heinich are robust with respect to perturbations of the cost function.
\end{GanSw}
\newtheorem{BL}[GanSw]{Example}
\begin{BL} (Bi-linear costs)
We now turn to bi-linear costs; suppose $c:(\mathbb{R}^n)^m \rightarrow \mathbb{R}$ is given by $c(x_1,x_2,...,x_m) =\sum_{i\neq j}(x_i)^TA_{ij}x_j$ for $n$ by $n$ matrices $A_{ij}$.  If $A_{1m}$ is non-singular, $c$ is $(1,m)$-twisted and $(1,m)$-non-degenerate.  Now, the Hessian terms in $T$ vanish and so the condition $T<0$ becomes a condition on the $A_{ij}$.  For example, when $m=3$, we have $T =A_{21}(A_{31})^{-1}A_{32}$; $T <0$ is the same condition that ensures the solution to \textbf{K} is contained in an $n$ dimensional submanifold in \cite{P}.  

Note that after changing coordinates in $x_2$ and $x_3$, we can assume any bi-linear three-marginal cost is of the form 
\begin{equation*}
c(x_1,x_2,x_3)=x_1 \cdot x_2 + x_1 \cdot x_3+ x_2^T A x_3
\end{equation*}
In these coordinates, the threefold product $A_{21}(A_{31})^{-1}A_{32}=A^T$.  Applying the linear change of coordinates
\begin{eqnarray*}
x_1 \mapsto U_1x_1 \\ 
x_2 \mapsto U_2x_2\\
x_3 \mapsto U_3x_3
\end{eqnarray*}
yields
\begin{equation*}
c(x_1,x_2,x_3)=x_1^TU_1^T U_2x_2 + x_1^TU_1^T U_3x_3+ x_2^TU_2^T A U_3x_3
\end{equation*}
If $A$ is negative definite and symmetric, then we can choose $U_3=U_2$ such that $U_2^T A U_3=-I$ and $U_1=-(U_2^T)^{-1}$ to obtain
\begin{equation*}
c(x_1,x_2,x_3)=-x_1^Tx_2 - x_1^Tx_3- x_2^Tx_3
\end{equation*}
which is equivalent \footnote{We say cost functions $c$ and $\overline{c}$ are equivalent if $\overline{c}(x_1,x_2,...,x_m) = c(x_1,x_2,...,x_m) +\sum_{i=1}^m g_i(x_i)$.  As the effect of the $g_i$'s is to shift the functionals $C(G_2,G_3,...,G_m)$ and $C(\mu)$ by the constant $\sum_{i=1}^{m}\int_{M_i}g_i(x_i)d\mu_i$, studying $c$ is essentially equivalent to studying $\overline{c}$.} to the cost of Gangbo and \'Swi\c ech.  As the symmetry of $D^2_{x_2x_1}c(D^2_{x_3x_1}c)^{-1}D^2_{x_3x_2}c$ is independent of our choice of coordinates, we conclude that $c$ is equivalent to Gangbo and \'Swi\c ech's cost if and only if $A_{21}(A_{31})^{-1}A_{32}$ is symmetric and negative definite.  Thus, when $m=3$ our result restricted to bi-linear costs generalizes Gangbo and \'Swi\c ech's theorem from costs for which $A_{21}(A_{31})^{-1}A_{32}$ is symmetric and negative definite to ones for which it is only negative definite. 
\end{BL}

\newtheorem{3M}[GanSw]{Example}
\begin{3M}
 There is another class of three marginal problems which Theorem \ref{mon} applies to: on $\mathbb{R}^n \times \mathbb{R}^n \times \mathbb{R}^n$, set 
\begin{equation*}
c(x_1,x_2,x_3)=g(x_1,x_3)+\frac{|x_1-x_2|^2}{2} +\frac{|x_3-x_2|^2}{2}.
\end{equation*}
If $g(x_1,x_3)=\frac{|x_1-x_3|^2}{2}$, this is equivalent to the cost of Gangbo and \'Swi\c ech.  More generally, if $g$ is $(1,3)$-twisted and non-degenerate, then $c$ is as well.  Moreover, if we make the usual identification between tangent spaces at different points in $\mathbb{R}^n$, we have
\begin{equation*}
 T_{\vec{y},\vec{y}(2)}=\big(D^2_{x_1x_3}g(y_1,y_3)\big)^{-1}.
\end{equation*}
Hence, if $D^2_{x_1x_3}g(y_1,y_3) <0$, we have $T_{\vec{y},\vec{y}(2)}<0$.  This will be the case if, for example, $g(x_1,x_3)=h(x_1-x_3)$ for $h$ uniformly convex or $g(x_1,x_3)=h(x_1+x_3)$ for $h$ uniformly concave.
\end{3M}

\newtheorem{hedonic}[GanSw]{Example}
\begin{hedonic}(Hedonic Pricing) Chiappori, McCann and Nesheim \cite{CMN} and Carlier and Ekeland \cite{CE} showed that finding equilibrium in a certain hedonic pricing model is equivalent to solving a multi-marginal optimal transportation problem with a cost function of the form $c(x_1,x_2,...,x_m) = \inf_{z \in Z}\sum_{i=1}^m f_i( x_i,z)$.  Let us assume:
\begin{enumerate}
\item $Z$ is a $C^2$ smooth $n$-dimensional manifold.
\item For all $i$,  $f_i$ is $C^{2}$ and non-degenerate. 
\item For each $(x_1,x_2,...,x_m)$ the infinum is attained by a unique $z(x_1,x_2,...,x_m) \in Z$ and
\item $\sum_{i=1}^m D^2_{zz}f_i(x_i,z(x_1,x_2,...,x_m))$ is non-singular.
\end{enumerate}
In \cite{P}, we showed that these conditions implied that $c$ is $C^2$ and $(i,j)$-non-degenerate for all $i \neq j$; we then showed that the support of any optimizer is contained in an $n$-dimensional Lipschitz submanifold of the product $M_1 \times M_2 \times ...M_m$.  Here we examine conditions on the $f_i$ that ensure the hypotheses of Theorem \ref{mon} are satisfied.  If, for fixed $i \neq j$, we assume in addition that:
\begin{enumerate}
\setcounter{enumi}{4}
 \item $f_i$ is $x_i, z$ twisted (that is, $z \mapsto D_{x_i}f_i(x_i,z)$ is injective) and
\item $f_j$ is $z,x_j$ twisted.
\end{enumerate}
then $c$ is $(i,j)$-twisted.  Indeed, note that $c(x_1,x_2,...,x_m) \leq \sum_{i=1}^m f_i( x_i,z))$ with equality when $z=z(x_1,x_2,...,x_m)$; therefore, 
\begin{equation}\label{max}
 D_{x_i}c(x_1,x_2,...,x_m)=D_{x_i}f(x_i, z(x_1,x_2,...,x_m))
\end{equation}
Therefore, for fixed $x_k$ for all $k \neq j$, the map $x_j \mapsto  D_{x_i}c(x_1,x_2,...,x_m)$ is the composition of the maps $x_j \mapsto z(x_1,x_2,...,x_m)$ and $z \mapsto D_{x_i}f(x_i, z)$.  The later map is injective by assumption. Now, note that 
\begin{equation*}
 \sum_{k=1}^m D_{z}f_i(x_i,z(x_1,x_2,...,x_m))=0;
\end{equation*}
 hence, 
\begin{equation*}
D_zf_j(x_j,z(x_1,x_2,...,x_m))=-\sum_{k \neq j} D_{z}f_k(x_k,z(x_1,x_2,...,x_m)).
\end{equation*}
  Twistedness of $f_j$ now immediately implies injectivity of the first map.
 
We now investigate the form of the tensor $T$.

 As $A(x_1,x_2,...,x_m):=\sum_{i=1}^m D^2_{zz}f_i(x_i,z(x_1,x_2,...,x_m))$ is non-singular by assumption, the implicit function theorem implies that $z(x_1,x_2,...,x_m)$ is differentiable and
\begin{equation*} 
D_{x_i}z(x_1,x_2,...,x_m)=-\big(A(x_1,x_2,...,x_m)\big)^{-1}D^2_{zx_i}f_i(x_i,z(x_1,x_2,...,x_m))
 \end{equation*}
Furthermore, note that as $A$ is positive semi-definite by the minimality of $z \mapsto \sum_{i=1}^m f_i( x_i,z))$ at $z(x_1,x_2,...,x_m)$, the non-singular assumption implies that it is in fact positive definite.

Differentiating (\ref{max}) with respect to $x_i$ for $i=2,3,..m-1$ yields:
\begin{equation*}
\hess{i}c=-(D^2_{x_iz}f_i)D_{x_i}z + \hess{i}f_i=-(D^2_{x_iz}f_i)A^{-1}(D^2_{zx_i}f_i) + \hess{i}f_i.
\end{equation*}
where we have suppressed the arguments $x_1,x_2,..x_m$ and $z(x_1,x_2,...,x_m)$. A similar calculation yields, for all $i \neq j$,
\begin{equation*}
D^2_{x_ix_j}c=(D^2_{x_iz}f_i)D_{x_j}z=-(D^2_{x_iz}f_i)A^{-1}(D^2_{zx_j}f_i)
\end{equation*}
Thus, for all $i \neq j$, a straightforward calculation yields
\begin{eqnarray*}
 D^2_{x_ix_m}c(D^{2}_{x_1x_m}c)^{-1}D^2_{x_1x_j}c=-(D^2_{x_iz}f_i)A^{-1}(D^2_{zx_j}f_i)=D^2_{x_ix_j}c,
\end{eqnarray*}
Hence, $S_{\vec{y}}$ is block diagonal.  Furthermore, another simple calculation implies that its $i$th diagonal block is
\begin{equation*}
  \Big[D^2_{x_ix_m}c(D^{2}_{x_1x_m}c)^{-1}D^2_{x_1x_i}c\Big]\Big(\vec{y}\Big)=-\Big[(D^2_{x_iz}f_i)A^{-1}(D^2_{zx_i}f_i)\Big]\Big(\vec{y},z(\vec{y})\Big).
\end{equation*}
In addition, $H_{\vec{y},\vec{y}(2),\vec{y}(3),...,\vec{y}(m-1)}$ is block diagonal and its $i$th block is 
\begin{eqnarray*}
-\Big[(D^2_{x_iz}f_i)A^{-1}(D^2_{zx_i}f_i)\Big]\Big(\vec{y}(i),z\big(\vec{y}(i)\big)\Big)+ \hess{i}f_i\Big(y_i,z\big(\vec{y}(i)\big)\Big)\\
\text{      }+\Big[(D^2_{x_iz}f_i)A^{-1}(D^2_{zx_i}f_i)\Big]\Big(\vec{y},z(\vec{y})\Big)-\hess{i}f_i\Big(y_i,z(\vec{y})\Big)
\end{eqnarray*}
Hence, $T_{\vec{y},\vec{y}(2),\vec{y}(3),...,\vec{y}(m-1)}$ is block diagonal and its $i$th block is 
\begin{equation}\label{T}
 -\Big[(D_{x_iz}f_i)A^{-1}(D^2_{zx_i}f_i)\Big]\Big(\vec{y}(i),z\big(\vec{y}(i)\big)\Big)+\hess{i}f_i\Big(y_i,z\big(\vec{y}(i)\big)\Big)-\hess{i}f_i\Big(y_i,z(\vec{y})\Big)
\end{equation}
Therefore, $T_{\vec{y},\vec{y}(2),\vec{y}(3),...,\vec{y}(m-1)}$ is negative definite if and only if each of its diagonal blocks is. Now, $A$ is symmetric and positive definite; therefore $A^{-1}$ is as well.  The first term in the $i$th block of (\ref{T}) is therefore negative definite; the entire block will be negative definite if this term dominates the difference of the Hessian terms.  This is the case if, for example, $M_i=\mathbb{R}^n$ and $f_i$ takes the form $f_i(x_i,z)=x_i\alpha_i(z)+\beta_i(x_i)+\lambda_i(z)$ for all $i=2,3,...,m-1$, in which case $\hess{i}f_i\Big(y_i,z\big(\vec{y}(i)\big)\Big)=\hess{i}f_i\Big(y_i,z(\vec{y})\Big)$.
\end{hedonic}

\end{document}